\documentclass{endm}
\usepackage{endmmacro}
\usepackage{graphicx} 
\usepackage{latexsym,epsfig}
\usepackage[usenames,dvipsnames]{color} 
\usepackage{url} 
\usepackage{soul} 
\usepackage{array} 
\usepackage{amssymb}
\usepackage{amsfonts}
\usepackage{amsmath}
\usepackage[english]{babel}
\usepackage{verbatim}
\usepackage{tikz}
\usepackage{subfigure} 
\usepackage{hyperref}

\newcommand{\F}{\mathcal{F}}
\newcommand{\B}{\mathcal{B}}

\newcommand{\C}{\mathbb{C}} 

\newcommand{\EndProof}{\hspace{\stretch{1}} $\Box$}

\makeatletter
  \newcommand\binomialCoefficient[2]{%
     \c@pgf@counta=#1
     \c@pgf@countb=#2
     \c@pgf@countc=\c@pgf@counta%
     \advance\c@pgf@countc by-\c@pgf@countb%
     \ifnum\c@pgf@countb>\c@pgf@countc%
         \c@pgf@countb=\c@pgf@countc%
     \fi%
     \c@pgf@countc=1
     \c@pgf@countd=0
     \pgfmathloop
         \ifnum\c@pgf@countd<\c@pgf@countb%
         \multiply\c@pgf@countc by\c@pgf@counta%
         \advance\c@pgf@counta by-1%
         \advance\c@pgf@countd by1%
         \divide\c@pgf@countc by\c@pgf@countd%
     \repeatpgfmathloop%
     \the\c@pgf@countc%
 }
\makeatother

\begin{document}

\begin{verbatim}\end{verbatim}\vspace{2.5cm}


\begin{frontmatter}

\title{Asymptotic Analysis of Central \\ Binomiacci Numbers}  

\author{Hebert P\'erez-Ros\'es\thanksref{myemail}
} 
\address{Department of Computer Science and Mathematics\\Rovira i Virgili University\\Tarragona, Spain\\
Conjoint fellow, University of Newcastle, Australia}

\thanks[myemail]{Email: \href{mailto:hebert.perez@matematica.udl.cat} 
{\texttt{\normalshape hebert.perez@urv.cat}}} 


\begin{abstract}
We make an asymptotic analysis via singularity analysis of generating functions of a number sequence that involves the Fibonacci numbers and generalizes the binomial coefficients. 
\end{abstract}


\end{frontmatter}


\section{Introduction}
\label{sec:intro}

Let $\langle \F_j \rangle_{j \geq 0} = \langle 1, 1, 2, 3, 5, 8, 13, \ldots \rangle$ be the sequence of Fibonacci numbers. We consider the bivariate number sequence $\langle \B_{k,n} \rangle_{k,n \geq 0}$ defined by the following recurrence equation: 

\begin{equation}
\label{eq:recurrence1}
\B_{k,n} = 
\begin{cases} 
\F_n, \; \mbox{ if } k=0 \mbox{ and } n \geq 0, \\
\F_k, \; \mbox{ if } n=0 \mbox{ and } k \geq 0, \\
\B_{k,n-1} + \B_{k-1,n}, \mbox{ otherwise } 
\end{cases} 
\end{equation}

In other words, the sequence $\langle \B_{k,n} \rangle$ is a generalization of the binomial coefficients, and it can be constructed with the aid of a Pascal-like table, where the ones at the boundaries have been replaced by the elements of $\langle \F_j \rangle$. It corresponds to sequence \href{https://oeis.org/A074829}{\textbf{A074829}} in the Online Encyclopedia of Integer Sequences \cite{oeis}, with the difference that we have organized the numbers in table form, rather than triangle form, for our convenience. The first few elements of $\langle \B_{k,n} \rangle$ are depicted in Figure \ref{tab:binomiacci1}, up to $k=4$ and $n=8$. 

\begin{table}[htp]
\label{tab:binomiacci1} 
\begin{center}
\begin{tabular}{|cc|*{9}{c}|} \hline
\multicolumn{2}{|c|}{} & \multicolumn{9}{c|}{$n$} \\
\multicolumn{2}{|c|}{$k$} & 0 & 1 & 2 & 3 & 4 & 5 & 6 & 7 & 8 \\ 
\hline 
\multicolumn{2}{|c|}{0} & \textbf{1} & 1  & 2  & 3  & 5   & 8   & 13  & 21  & 44 \\ 
\multicolumn{2}{|c|}{1} & 1  & \textbf{2} & 4  & 7  & 12  & 20  & 33  & 54  & 98 \\ 
\multicolumn{2}{|c|}{2} & 2  & 4  & \textbf{8} & 15 & 27  & 47  & 80  & 134 & 232 \\ 
\multicolumn{2}{|c|}{3} & 3  & 7  & 15 & \textbf{30} & 57  & 104 & 184 & 318 & 550 \\ 
\multicolumn{2}{|c|}{4} & 5  & 12 & 27 & 57 & \textbf{114} & 218 & 402 & 720 & 1270 \\ 
\hline
\end{tabular}
\caption{The Binomiacci table}
\end{center}
\end{table}

\begin{table}[htp]
\label{tab:binomiacci2} 
\begin{center}
\begin{tabular}{>{}l<{\hspace{12pt}}*{13}{c}}
&&&&&&&\textbf{1}&&&&&&\\
&&&&&&1&&1&&&&&\\
&&&&&2&&\textbf{2}&&2&&&&\\
&&&&3&&4&&4&&3&&&\\
&&&5&&7&&\textbf{8}&&7&&5&&\\
&&8&&12&&15&&15&&12&&8&\\
&13&&20&&27&&\textbf{30}&&27&&20&&13
\end{tabular}
\caption{The Binomiacci triangle}
\end{center}
\end{table}

Since the numbers $\B_{k,n}$ are a combination of Fibonacci numbers and binomial coefficients, we may feel tempted to name them \lq Fibonomial coefficients\rq. Unfortunately, that name is already taken, and it corresponds to another number sequence that is well studied (see, for instance \cite{Ben09,Krot04,Tro07}). Nevertheless, the alternative portmanteau term \lq \emph{Binomiacci}\rq \ does not appear in the literature, and seems to be a good option for our sequence \href{https://oeis.org/A074829}{\textbf{A074829}}. 

Apparently, the idea of combining Fibonacci numbers with a Pascal-like triangle or table dates back to Hosoya \cite{Hoso76}. More recently other authors have been motivated by this idea \cite{Bel23} \cite{Son22}  and have defined general sequences of numbers that include our sequence of Binomiacci numbers \href{https://oeis.org/A074829}{\textbf{A074829}} as a special case, obtaining useful identities for them. 

Nevertheless, Binomiacci numbers still do not seem to be well-studied. In particular, the diagonal elements have not been investigated yet, as far as we know. In this paper we obtain generating functions for $\langle \B_{k,n} \rangle$, as well as for its diagonal, and with their aid, we obtain asymptotic estimates for the diagonal elements, the central Binomiacci numbers $\langle \B_{n,n} \rangle$. 

\section{Generating functions}
\label{sec:generatingfunctions}

Let us now state our main result, concerning the generating functions of Binomiacci numbers: 
\begin{lemma}
\label{lemma:genfuns} 
Let $A_k(z) = \sum_{n \geq 0} \B_{k,n} z^n$ be the univariate ordinary generating function associated with the sequence $\langle \B_{k,n} \rangle$, where $z \in \C$, and let $G(z,w) = \sum_{k \geq 0} \sum_{n \geq 0} \B_{k,n} z^n w^k$ be the corresponding bivariate generating function, where $z,w \in \C$. Then 
\begin{equation}
\label{eq:genfun1}
A_k(z) = 
    \begin{cases} 
        \frac{1}{1-z-z^2}, \; \mbox{ if } k=0 \mbox{ (the Fibonacci numbers), } \\
        \frac{1}{(1-z)(1-z-z^2)} \; \mbox{ if } k=1, \\
        \frac{1}{(1-z)(1-z-z^2)} + \sum_{i=0}^{k-2} \frac{F_i}{(1-z)^{k-1-i}}, \mbox{ for } k \geq 2, 
    \end{cases} 
\end{equation}
and 
\begin{equation}
\label{eq:genfun2}
    G(z,w) = \frac{1 - z - w + zw - z^2 w^2}{(1-z-z^2)(1-w-w^2)(1-z-w)} 
\end{equation}
\end{lemma}
\begin{proof} 
In order to obtain $A_k(z)$ we start by multiplying the recurrence $\B_{k+1,n+1} = \B_{k,n+11} + \B_{k+1,n}$ by $z^n$ and summing over $n \geq 0$ to get 
\begin{displaymath}
\label{eq:genfun1A} 
\begin{split}
    \frac{A_{k+1}(z)-\B_{k+1,0}}{z} &= \frac{A_{k}(z)-\B_{k,0}}{z} + A_{k+1}(z), \\ 
    A_{k+1}(z)-\B_{k+1,0} &= A_{k}(z)-\B_{k,0} + z A_{k+1}(z), \\ 
    (1-z)A_{k+1}(z) &= A_{k}(z)+\B_{k+1,0}-\B_{k,0}, \\
                &= A_{k}(z)+\F_{k+1}-\F_k \\
                &= A_{k}(z)+\F_{k-1}, \mbox{ whence }\\
    A_{k+1}(z) &= \frac{A_{k}(z)}{1-z} + \frac{\F_{k-1}}{1-z}, 
\end{split}
\end{displaymath}
from where it is now easy to derive (\ref{eq:genfun1}). 

Now let $G(z,w) = \sum_{k \geq 0} A_k(z) w^k$. We multiply (\ref{eq:genfun1}) by $w^k$ and sum over $k \geq 2$ to get
\begin{equation}
\label{eq:genfun2A} 
    G(z,w) = \sum_k \frac{w^k}{(1-z)^k(1-z-z^2)} + \sum_k \left( \sum_{i=0}^{k-2} \frac{F_i}{(1-z)^{k-1-i}} \right) w^k, 
\end{equation}
The first term of (\ref{eq:genfun2A}) is 
\begin{displaymath}
\begin{split}
    B(z,w) &= \frac{1}{1-z-z^2} \sum_k \frac{w^k}{(1-z)^k} \\
       &= \frac{1}{(1-z-z^2)(1-\frac{w}{1-z})} \\
       &= \frac{1-z}{(1-z-z^2)(1-z-w)}, 
\end{split}
\end{displaymath}
while the second term is 
\begin{displaymath}
    C(z,w) = \frac{w^2}{1-z} \sum_k \left( \sum_{i=0}^{k-2} \frac{F_i}{(1-z)^{k-2-i}} \right) w^{k-2}.  
\end{displaymath}
The inner sum of $C(z,w)$ is obviously the convolution of the sequences $\langle \F_0, \ldots, \F_{k-2} \rangle$ and $\langle 1, \frac{1}{1-z}, \ldots, \frac{1}{(1-z)^{k-2}} \rangle$, with respective generating functions $\frac{1}{1-z-z^2}$ and $\frac{1-z}{1-z-w}$, hence 
\begin{displaymath}
\begin{split}
C(z,w) &= \frac{w^2(1-z)}{(1-z)(1-z-z^2)(1-w-w^2)(1-z-w)} \\
       &= \frac{w^2}{(1-z-z^2)(1-w-w^2)(1-z-w)},  
\end{split}   
\end{displaymath}
and 
\begin{displaymath}
\begin{split}
    G(z,w) &= \frac{1-z}{(1-z-z^2)(1-z-w)} + \frac{w^2}{(1-z-z^2)(1-w-w^2)(1-z-w)} \\
            &= \frac{1 - z - w + zw - z^2 w^2}{(1-z-z^2)(1-w-w^2)(1-z-w)}.  
\end{split}   
\end{displaymath}
\end{proof}

As in the Pascal triangle, we can see that the largest elements of Table \ref{tab:binomiacci1} are located on the diagonal. These are our main focus of interest, and our main result deals with finding a generating function for the diagonal elements, also called \emph{central elements}, which can be done with the aid of our bivariate generating function obtained in Lemma \ref{eq:genfun1}. 

\begin{theorem}
\label{theo:diagonal-genfun1}
Let $\displaystyle C(s) = \sum_{n \geq 0} \B_{n,n} s^n$ denote the univariate ordinary generating function associated with the sequence of the central Binomiacci numbers $\langle \B_{n,n} \rangle_{n=0}^{\infty}$, where $s \in \C$. Then, 
\begin{equation}
\label{eq:diagonal-genfun1}
    C(s) = \frac{(s-1)(\sqrt{1-4s}-s)}{(s^2+4s-1)\sqrt{1-4s}}. 
\end{equation} 
\end{theorem}
\begin{proof} 
Stanley (\cite{Stan99}, Section 6.3) shows how to obtain a generating function for the diagonal elements of a  multivariate generating function. In fact, he gives two approaches: an algebraic approach and an analytic one, and he exemplifies both methods with the central binomial coefficients. Here we will follow the latter approach. The procedure is quite similar to the one described in \cite{Stan99}, though our calculations are much more involved. 

The crux of the argument is that the generating function of the diagonal of $G(z,w)$ is the constant term of $G(z,\frac{s}{z})$, regarded as a Laurent series in $z$, whose coefficients are power series in $s$. 

Hence, 
\begin{equation}
    \mbox{diag } G = [z^0] G(z,\frac{s}{z}) = \frac{1}{2\pi i} \int_{|z|=\rho} \frac{1}{z} G(z,\frac{s}{z}) dz. 
\end{equation}
Now, by the Residue Theorem, 
\begin{equation}
    \int_{|z|=\rho} \frac{1}{z} G(z,\frac{s}{z}) dz = \sum_{z=z(s)} Res_z \frac{1}{z} G(z,\frac{s}{z}),  
\end{equation}
where the sum ranges over all singularities of $\frac{1}{z} G(z,\frac{s}{z})$ inside the circle $|z|=\rho$, with radius $\rho > 0$, which are the singularities that satisfy $\displaystyle \lim_{s \rightarrow 0} z(s) = 0$. Thus, 
\begin{equation}
     \frac{1}{z} G(z,\frac{s}{z}) = \frac{z(z-s-z^2+zs-zs^2)}{(1-z-z^2)(z-s-z^2)(z^2-zs-s^2)}
\end{equation}
The factor $1-z-z^2$ in the denominator yields the singularities $z = \frac{1}{2}(\pm \sqrt{5}-1)$, but these are fixed, i.e. independent from $s$. Next, the factor $z-s-z^2$ yields the singularities $z = \frac{1}{2}(1 \pm \sqrt{1-4s})$. Of these, only $z_1 = \frac{1}{2}(1 - \sqrt{1-4s})$ tends to zero as $s \rightarrow 0$. Finally, the factor $z^2-zs-s^2$ yields the singularities $z_{2,3} = \frac{1}{2}(1 \pm \sqrt{5})s$, which both tend to zero as $s \rightarrow 0$. 

It turns out that all our singularities are simple poles. Recall that if $z_0$ is a simple pole of the function $F(z) = \frac{P(z)}{Q(z)}$, then 
\begin{equation}
    Res(F, z_0) = \frac{P(z_0)}{Q'(z_0)}.   
\end{equation}
In the case of our first pole, $z_1 = \frac{1}{2}(1 - \sqrt{1-4s})$, the function $F(z) = \frac{1}{z} G(z,\frac{s}{z})$ is of the form $\frac{P(z)}{Q(z)}$, with $\displaystyle P(z) = \frac{z(z-s-z^2+zs-zs^2)}{(1-z-z^2)(z^2-zs-s^2)}$ and $Q(z) = z-s-z^2$. Hence, $Q'(z) = 1-2z$, and 
\begin{align*}
    Res(F, z_1) &= \frac{P(z_1)}{Q'(z_1)} = \frac{z_1(z_1-s-z_1^2+z_1s-z_1s^2)}{(1-z_1-z_1^2)(1-2z_1)(z_1^2-z_1s-s^2)} \\
                &= \frac{s(1-s)}{\sqrt{1-4s}\left( s^2+4s-1 \right)}. 
\end{align*}
Similarly, in our second pole, $z_{2} = \frac{1}{2}(1 +\sqrt{5})s$, we also have a function of the form $\frac{P(z)}{Q(z)}$, but this time $\displaystyle P(z) = \frac{z(z-s-z^2+zs-zs^2)}{(1-z-z^2)(z-s-z^2)}$ and $Q(z) = z^2-sz-s^2$, so $Q'(z) = 2z-s$. Thus, 
\begin{align*}
    Res(F, z_2) &= \frac{P(z_2)}{Q'(z_2)} = \frac{z_2(z_2-s-z_2^2+z_2s-z_2s^2)}{(1-z_2-z_2^2)(z_2-s-z_2^2)(2z_2-s)} \\
                &= \frac{-2}{\sqrt{5}\left[ (3+\sqrt{5})s-\sqrt{5}+1 \right]}. 
\end{align*}
Finally, in our third pole $z_{3} = \frac{1}{2}(1 -\sqrt{5})s$ we also have $\displaystyle P(z) = \frac{z(z-s-z^2+zs-zs^2)}{(1-z-z^2)(z-s-z^2)}$, $Q(z) = z^2-sz-s^2$, and $Q'(z) = 2z-s$. So, 
\begin{align*}
    Res(F, z_3) &= \frac{P(z_3)}{Q'(z_3)} = \frac{z_3(z_3-s-z_3^2+z_3s-z_3s^2)}{(1-z_3-z_3^2)(z_3-s-z_3^2)(2z_3-s)} \\
                &= \frac{2}{(3\sqrt{5}-5)s+\sqrt{5}+5 }. 
\end{align*}
After adding up the three residues and simplifying the expression, we get 
\begin{displaymath}
    Res(F, z_1)+Res(F, z_2)+Res(F, z_3) = \frac{(s-1)(\sqrt{1-4s}-s)}{(s^2+4s-1)\sqrt{1-4s}}, 
\end{displaymath}
as desired. 

\end{proof}

\section{Asymptotic Analysis}
\label{sec:asymptotic}

Now we can extract asymptotic information for the central Binomiacci numbers $\B_{n,n}$ from their generating function $C(s)$, obtained in Theorem \ref{theo:diagonal-genfun1}. Again, this involves investigating the singularities of $C(s)$. 

We readily recognize that $C(s)$ has singularities at $s = -2 \pm \sqrt{5}$, and $s = \frac{1}{4}$. It turns out that all the singularities are real, and the one with smallest modulus is $\sqrt{5}-2 \approx 0.23607$. However, this is a removable  singularity and therefore does not affect the growth of $\B_{n,n}$. 

The removable singularity is followed very closely by a branch point at $s = \frac{1}{4}$. Branch points and poles are algebraic singularities. Given a univariate function of a complex variable $f(s)$, recall that $\alpha$ is an \emph{algebraic singularity} of $f$ if $f$ can be written near $\alpha$ as 

\begin{equation}
\label{eq:singularity}
f(s) = f_0(s) + \frac{g(s)}{(1-s/\alpha)^\omega}
\end{equation}

where $f_0$ and $g$ are analytic near $\alpha$, $g$ is nonzero near $\alpha$, and $\omega$ is a real number different from $0, -1, -2, \ldots$. 

Indeed, $C(s)$ can be written in the above form, since 

\begin{displaymath}
    C(s) = \frac{s-1}{s^2+4s-1} + \frac{s-s^2}{(s^2+4s-1) \sqrt{1-4s}},  
\end{displaymath} 

where $\displaystyle f_0(s) = \frac{s-1}{s^2+4s-1}$ and $\displaystyle g(s) = \frac{s-s^2}{(s^2+4s-1)}$ are analytic near $\alpha = \frac{1}{4}$, and $\omega = \frac{1}{2}$. 

Now we can easily obtain a coarse estimate of the growth of $\B_{n,n}$ by applying Theorem 3 of \cite{Lue80}, that we reproduce here: 

\begin{theorem}
\label{th:asymptotic1}
Suppose that for some real $\rho > 0$, $f(s)$ is analytic in the region $|s| < \rho$, and has a finite number $\tau >0$ of singularities on the circle $|s| = \rho$, all of which are algebraic. Let $\alpha_i, \ \omega_i$, and $g_i$ be the values of $\alpha, \ \omega$, and $g$ in (\ref{eq:singularity}), corresponding to the $i$-th such singularity. Then $f(s)$ is the generating function for a sequence $\langle a_n \rangle$ satisfying
\begin{equation}
    \label{eq:asymptotic1}
    a_n = \frac{1}{n} \sum_{i=1}^{\tau} \frac{g_i(\alpha_i)n^{\omega_i}}{\Gamma(\omega_i)\alpha_i^n} + o(\rho^{-n}n^{\Omega-1})
\end{equation}
where $\Omega$ is the maximum of the $\omega_i$ and $\Gamma$ denotes the Gamma function.
\end{theorem}

In our case there is only one algebraic singularity on the circle $|s| = \frac{1}{4}$, hence $\tau=1$, $\rho = \alpha_1 = \frac{1}{4}$, $\Omega = \omega_1 = \frac{1}{2}$ and $g(\frac{1}{4}) = 3$. By plugging all the parameters into Equation \ref{eq:asymptotic1} we get 

\begin{corollary}
    \label{coro:asymptotics1}
    \begin{displaymath}
        \B_{n,n} = \frac{3 \cdot 4^n}{\sqrt{\pi n}} + o\left( \frac{4^n}{\sqrt{n}} \right) 
    \end{displaymath}
\EndProof     
\end{corollary}



In Table \ref{tab:asymptotics1} we can compare the values of the asymptotic approximation $\displaystyle \frac{3 \cdot 4^n}{\sqrt{\pi n}}$ given by Corollary \ref{coro:asymptotics1} with the actual values of $\B_{n,n}$ for $n= 1, \ldots, 15$. As we can see in the table, both quantities differ significantly for the first values of $n$, but as $n$ grows the ratio between them decreases slowly towards one. 

\begin{table}[h!]
\centering
\begin{tabular}{||c c c c||} 
 \hline
 $n$ & $\B_{n,n}$ & $\displaystyle \frac{3 \cdot 4^n}{\sqrt{\pi n}}$ & ratio \\ [0.8ex] 
 \hline\hline
 1 & 2 & 6.77 & 3.385 \\ 
 2 & 8 & 19.15 & 2.39 \\
 3 & 30 & 62.54 & 2.08 \\
 4 & 114 & 216.65 & 1.9 \\
 5 & 436 & 775.1 & 1.77 \\ 
 6 & 1676 & 2830.3 & 1.688 \\
 7 & 6468 & 10481.4 & 1.62 \\
 8 & 25040 & 39217.6 & 1.56 \\ 
 9 & 97190 & 147899 & 1.52 \\
 10 & 378050 & 561237 & 1.48 \\
 11 & 1473254 & 2140470 & 1.45 \\
 12 & 5750390 & 8197390 & 1.425 \\
 13 & 22476090 & 31503200 & 1.4 \\
 14 & 87958306 & 121429000 & 1.38 \\
 15 & 344593314 & 469246000 & 1.36 \\
 [1ex] 
 \hline
\end{tabular}
\caption{Coarse asymptotic approximation of $\B_{n,n}$}
\label{tab:asymptotics1} 
\end{table}

\section*{Acknowledgements}
\label{sec:disclaimer}

The author was partially supported by Grant 2021 SGR 00115 from the Government of Catalonia, by Project ACITECH PID2021-124928NB-I00, funded by MCIN/AEI/ 10.13039/501100011033/FEDER, EU, and by Project HERMES, funded by INCIBE and by the European Union NextGeneration EU/PRTR.  


\end{document}